\documentclass[a4paper,11pt,reqno,twoside]{article}
\usepackage[usenames,dvipsnames]{color} 
\usepackage{amsmath,amsthm,amsfonts,amssymb,amscd,mathrsfs}
\usepackage[colorlinks,linkcolor=BlueViolet,citecolor=RedViolet,pdftex]{hyperref}
\usepackage{array} 
\usepackage{longtable} 
\usepackage{xfrac} 
\usepackage{verbatim} 
\usepackage{listings} 
\usepackage{bbm} 
\usepackage{tikz} 
\usetikzlibrary{calc}
\usepackage{graphicx}
\usepackage{MnSymbol} 
\usepackage{esvect} 
\usepackage{upgreek} 
\usepackage{yfonts} 
\usepackage{relsize} 
\usepackage[top=2cm, bottom=2cm, left=1cm, right=1cm]{geometry}
\usepackage{titlesec} 
\usepackage{caption} 
\usepackage{shadethm} 
\usepackage[tikz]{bclogo} 
\usepackage{fancyhdr} 
\usepackage{enumitem} 

\def\vs{\vskip 0.30cm}
\def\vss{\vskip 0.05cm}

\def\mvs{\vspace{-0.2cm}}

\def\medmu{\medmuskip=0mu}

\def\ga{\upgamma}

\def\la{\uplambda}
\def\si{\upsigma}

\def\longto{\longrightarrow}

\newcommand{\inv}{^{-1}}

\newcommand{\deff}[1]{\textbf{#1}}

\newcommand{\N}{\mathbb{N}}

\newcommand{\Z}{\mathbb{Z}}

\newcommand{\sn}{S_n}

\newcommand{\aut}{\mathrm{Aut}}

\newcommand{\supp}{\mathrm{supp}}
\def\fix{\mathrm{fix}}

\newcommand{\gp}[1]{\left\langle #1 \right\rangle}

\newcommand{\set}[1]{\left\{ #1\right\}}

\newcommand{\sign}[1]{\mathrm{sign}(#1)}
\newcommand{\drob}[2]{{\,\Large \sfrac{#1}{#2}}}

\theoremstyle{plain}
\definecolor{shadethmcolor}{cmyk}{.02,.02,0,0} 
\definecolor{shaderulecolor}{cmyk}{.75,.75,0,.5} 

\newtheorem*{theore}{Theorem}
\newshadetheorem{theor}{Theorem}

\newtheorem{prop}{Proposition}

\newtheorem{Lemma}{Lemma}

\newtheorem{Corollary}{Corollary}

\theoremstyle{remark}

\newtheorem*{rema}{\textbf{Remark}}
\theoremstyle{definition}

\newdimen\arrowsize 
\pgfarrowsdeclare{arcs}{arcs} 
{ 
\arrowsize=0.2pt 
\advance\arrowsize by .5\pgflinewidth 
\pgfarrowsleftextend{-4\arrowsize-.5\pgflinewidth} 
\pgfarrowsrightextend{.5\pgflinewidth} 
} 
{ 
\arrowsize=0.2pt 
\advance\arrowsize by .3\pgflinewidth 
\pgfsetdash{}{0pt} 
\pgfsetroundjoin 
\pgfsetroundcap 
\pgfpathmoveto{\pgfpoint{-17\arrowsize}{5\arrowsize}} 
\pgfpatharc{180}{270}{17\arrowsize and 5\arrowsize} 
\pgfusepathqstroke 
\pgfpathmoveto{\pgfpointorigin} 
\pgfpatharc{90}{180}{17\arrowsize and 5\arrowsize} 
\pgfusepathqstroke 
}

\title{All even permutations with large support are commutators of generating pairs.}
\author{David Zmiaikou \\  
\emph{\small Max Planck Institute for Mathematics} \\ 
\emph{\small Vivatsgasse 7, 53111 Bonn, Germany}}

\titleformat{\section}{\Large\bfseries\scshape}{\thesection}{0.5cm}{}
\titleformat{\subsection}{\Large\bfseries\color{Mahogany}}{\large\thesubsection}{0.4cm}{}
\titlespacing{\section}{0cm}{1.2cm}{0.5cm}
\titlespacing{\subsection}{0cm}{1cm}{0.3cm}

\begin{document}
\maketitle
\abstract{We establish that any even permutation $\mu\in A_n$ moving at least $\left[\frac{3n}{4}\right]+o(n)$ points is the commutator of a generating pair of $A_n$ and a generating pair of $S_n$. From this we deduce an exponential lower bound for the number of systems of transitivity in alternating and symmetric groups.

\section{Introduction}
\noindent Denote by $A_n$ and $S_n$ the alternating and the symmetric groups of the set $\set{1, 2, \ldots, n}$. The action on the set is considered to be from the left, for instance $(1\;2)(2\;3)=(1\;2\;3)$. Due to \O ystein Ore \cite[1951]{ore}, we know that any element of the alternating group $A_n$ with $n\ge 5$ is a commutator of two elements. On the other hand, John D.~Dixon proved in \cite[1969]{dixon1} that the propotion of pairs of permutations from $S_n$ generating $A_n$ or $S_n$ is greater than $1- 2/(\log\log n)^2$ for sufficiently large $n$\footnote{It was an old conjecture of Eugen Netto \cite[1892]{netto} that the probability of a random pair of permutations to generate $A_n$ or $S_n$ tends to $1$ as $n\to\infty$.}.  
Therefore, a natural question arose which even permutations are actually commutators of generating pairs. 

In our paper, we are going to establish the next result:

\begin{theore}\label{th:alternatingorigami}
Let $n\ge 10$ be a positive integer, and let $p$ be a prime number such that $\left[\frac{3n}{4}\right] \le p \le n-3$. Then every permutation $\mu\in A_n$ moving at least $p+2$ points can be presented as commutators 
$$\mu=[\si,\tau_1] = [\si,\tau_2]\,,$$ 
where pairs $(\sigma, \tau_1)$ and $(\sigma, \tau_2)$ generate $A_n$ and $S_n$ respectively, $\si$ is a $p$-cycle and $\tau_1\inv \tau_2$ is a transposition.
\end{theore}

\begin{rema}
When $n\ge 14$ (except $n=19$) there always exists a prime $p$ such that $\left[ \frac{3n}{4} \right]\le p\le n-3$. Indeed, Jitsuro Nagura \cite[1952]{nagura} proved that, for any $m\ge 25$, the interval between $m$ and $\frac{6}{5}m$ contains a prime number. Therefore, for $\left[\frac{3n}{4}\right]\ge 25$ or equivalently $n\ge 34$, there is a prime $p$ such that\mvs
\begin{equation}\label{eq:primep}
\renewcommand{\arraystretch}{3}
\begin{array}{rl} 
\left[\dfrac{3n}{4}\right]\;\le\;\;\; p& \le\; \dfrac{6}{5}\left[\dfrac{3n}{4}\right]\le \dfrac{9n}{10}, \\
\textrm{and so }\; \left[\dfrac{3n}{4}\right]\; \le\;\;\; p& \le\; \left[\dfrac{9n}{10}\right] = n + \left[-\dfrac{n}{10}\right].
\end{array}
\renewcommand{\arraystretch}{1}
\end{equation}
For $n\in\lsem 14,33\rsem\setminus \{19\}$, a prime $p$ is given in Table \ref{table:primenumber}.
\begin{table}[htdp]
\caption{A prime number $p$ such that $\left[\frac{3n}{4}\right]\le p\le n-3$, where $n\in\lsem 14,33\rsem\setminus \{19\}$.}
\renewcommand{\arraystretch}{1.4}
$$
\begin{array}{|c|c|c|c|c|c|c|c|c|c|c|c|c|c|c|c|c|c|c|c|}
\hline
n&14&15&16&17&18&20&21&22&23&24&25&26&27&28&29&30&31&32&33\\ \hline
\left[\frac{3n}{4}\right]&10&11&12&12&13&15&15&16&17&18&18&19&20&21&21&22&23&24&24\\ \hline\hline
p&11&11&13&13&13&17&17&17&17&19&19&19&23&23&23&23&23&29&29\\\hline
\end{array}
$$
\label{table:primenumber}
\renewcommand{\arraystretch}{1}
\end{table}%
\end{rema}

As a corollary of Theorem, we obtain that any even permutation $\mu\in A_n$ moving at least $\left[\frac{3n}{4}\right]+o(n)$ points is the commutator of a generating pair of $A_n$ and a generating pair of $S_n$ (see Corollary \ref{cor:d10}b). The proportion of such permutation tends to $1$ as $n\to\infty$ (see Corollary \ref{cor:altrandomcomm}). We also deduce an exponential lower bound on the number of $T_2$-systems in alternating and symmetric groups (see Corollary \ref{cor:alternatingstrata}). The latter is related to the number of connected components in Product Replacement Algorithm graphs (see \cite{pak}) and the number of square-tiled surfaces in a given stratum with monodromy group $A_n$ or $S_n$ (see \cite{zmiaikou}).

\vs Let us mention that Martin J. Evans answered the question for even permutations of particular cycle types\footnote{A permutation is of \deff{cycle type} $r_1,\ldots,r_k$ if it is a product of disjoint non-trivial cycles of lengths $r_1,\ldots,r_k$.} in the PhD thesis \cite[1985]{evans}. More precisely, from the proof of his Theorem 3.2 one deduces the following:
\begin{itemize}
\item Let $\mu\in A_n$ be of cycle type $3, r_{1}, \ldots, r_{k}, r_{1}, \ldots, r_{k}$. Then $\mu$ is the commutator $[\sigma, (1\;2)\tau]$ of a generating pair of $S_n$ such that 
	\begin{itemize}
	\item[] $\sigma = (2\;3\;\ldots\;n)$ and $\tau$ is of cycle type $r_1,\ldots,r_k$ with\footnote{We denote by $\supp(\tau)$ and $\fix(\tau)$ the support and the set of fixed points of $\tau$ respectively.} $\supp(\tau)\subseteq 
	\set{4,6,\ldots,n-1}$, if $n\ge 7$ is odd;
	\item[] $\sigma = (1\;2\;\ldots\;n)$ and $\tau$ is of cycle type $r_1,\ldots,r_k$ with $\supp(\tau)\subseteq 
	\set{4,6,\ldots,n-2}$, if $n\ge 8$ is even.
	\end{itemize}

\item Let $n\ge 9$ be odd and $\mu\in A_n$ be of cycle type $3, r_{1}, \ldots, r_{k}, r_{1}, \ldots, r_{k}$.  Then $\mu$ is the commutator $[(1\;2\;\ldots\;n), \gamma\tau]$ of a generating pair of $A_n$ such that $\tau$ is of cycle type $r_1,\ldots,r_k$ with $\supp(\tau)\subseteq \set{6,8,\ldots,n-1}$ and
	\begin{itemize}
	\item[] $\gamma=(1\;2)$, if $\tau$ is odd;
	\item[] $\gamma=(1\;2\;3)$, if $\tau$ is even.
	\end{itemize}

\item Let $n\ge 8$ be even and $\mu\in A_n$ be of cycle type $3, r_{1}, \ldots, r_{k}, r_{1}, \ldots, r_{k}$\,, where $\prod_{i=1}^{k} (-1)^{r_i + 1} = -1$.  Then $\mu$ is the commutator $[(2\;3\;\ldots\;n), (1\;2)\tau]$ of a generating pair of $A_n$ such that $\tau$ is of cycle type $r_1,\ldots,r_k$ with $\supp(\tau)\subseteq \set{5,7,\ldots,n-1}$.

\item Let $n\ge 8$ be even and $\mu\in A_n$ be of cycle type $2,2, r_{1}, \ldots, r_{k}, r_{1}, \ldots, r_{k}$\,, where $\prod_{i=1}^{k} (-1)^{r_i + 1} = 1$.  Then $\mu$ is the commutator $[(2\;3\;\ldots\;n), (1\;2\;3)\tau]$ of a generating pair of $A_n$ such that $\tau$ is of cycle type $r_1,\ldots,r_k$ with $\supp(\tau)\subseteq \set{5,7,\ldots,n-1}$.
\end{itemize}
Also, Shelly Garion and Aner Shalev \cite[2009]{garion-shalev} extended related research on simple groups. They showed that in any finite simple non-abelian group almost all elements are commutators of generating pairs.


\section*{Acknowledgements}
\noindent I am grateful to Jean-Christophe Yoccoz for important remarks and suggestions. The results of this paper were obtained while I was working on my PhD thesis \cite{zmiaikou} under his supervision.

\section{Proof of Theorem}
\noindent We will need the following propositions and lemmas.

\begin{prop}[Edward Bertram \cite{bertram}, 1972]\label{prop:bertram}
Every even permutation $\mu\in A_n$ is a commutator $[\si,\tau]$, where $\si$ is an $\ell$-cycle and $\tau\in S_n$\,, if and only if\, $\left[\frac{3n}{4}\right]\le \ell\le n$.	
\end{prop}

Denote by $\lsem x, y\rsem$ the set of integers from $x$ to $y$. A nonempty subset $\Delta$ of the set $\Omega=\lsem 1, n\rsem$ is called a \deff{block} for a permutation group $G\subseteq\sn$ if, for each $\alpha\in G$, either $\alpha(\Delta)=\Delta$ or $\alpha(\Delta)\cap\Delta=\emptyset$. In particular, the singletons $\{x\}$ and the whole set $\Omega$ are blocks, which are called \deff{trivial}. A permutation group $G$ is said to be \deff{primitive} if it has no nontrivial blocks.

\begin{prop}[Camille Jordan \cite{jordan73}, 1873]\label{th:jordan}
Let $G$ be a primitive subgroup of $S_n$.
If $G$ contains a cycle of prime order $p\leqslant n-3$, then either $G=A_n$ or $G=S_n$.
\end{prop}

It is easy to see that a primitive permutation group is automatically transitive. The inverse is not always true when $n$ is composite: a transitive group can have a nontrivial block $\Delta$ and, if it does, the images $\alpha(\Delta)$ form a $G$-invariant partition of $\Omega$ in which all parts have equal size $1<|\Delta|<n$. In this case $|\Delta|$ must divide $n$. This observation can also be stated as the next lemma:
\begin{Lemma}\label{lem:blocks}
  Let $G$ be a transitive permutation group on $\Omega$, and $x\in \Omega$. Then $G$ is primitive if and only if the only blocks containing $x$ are $\{x\}$ and $\Omega$.
\end{Lemma}

\begin{Lemma}\label{lem:primitivelargecycle} 
If a transitive permutation group $G$ of degree $n$ contains a $p$-cycle for a prime $p>\frac{n}{2}$, then it is primitive.
\end{Lemma}
\begin{proof} 
Let $\sigma = (1\;2\;\ldots\;p)$ be a $p$-cycle in $G$. Consider a block $\Delta\subseteq \lsem 1,n\rsem$ for the transitive group $G$ such that $1\in\Delta$. If this block also contains an integer $x>p$, then $\si^m(\Delta)\cap\Delta\neq \emptyset$ for any $m\in\Z$, and so $\{1,2,\hdots,p\}\subset \Delta$. Thus $|\Delta|=n$, since $|\Delta|$ divides $n$ and $|\Delta|\ge p> n/2$. Suppose now that $\Delta\subseteq \lsem 1,p\rsem$. Then it is a block for the subgroup $\gp{\si}\subseteq G$ acting transitively on the set $\lsem 1,p\rsem$, and so $\Delta=\{1\}$ or $|\Delta|=p$. The latter is only impossible when $p=n$. In all cases we obtained that $|\Delta|=1$ or $n$. By Lemma \ref{lem:blocks} the group $G$ is primitive. 
\end{proof}

\begin{Lemma}\label{lem:sttransitive}
Let $p$ and $n$ be positive integers with $1<p<n$. Two permutations $\si = (1\;2\;\ldots\;p)$ and $\tau\in S_n$ generate a transitive subgroup if and only if $\tau$ fixes no integer in $\lsem p+1, n\rsem$ and each disjoint cycle of $\tau$ moves at least one integer from $\lsem 1, p\rsem$.

\end{Lemma}
\begin{proof} Consider a decomposition of $\tau$ into disjoint non-trivial cycles:
   $ \tau=\ga_1 \ga_2 \hdots \ga_r $.
   We want to show that the permutations $\sigma$ and $\tau$ generate a transitive permutation group if and only if 
\begin{equation}\label{eq:trans}
   \lsem p+1, n\rsem\subseteq \supp(\tau)\quad\textrm{and}\quad \lsem 1, p\rsem\cap \supp(\ga_i)\neq\emptyset\; \textrm{ for any } 1\le i\le r.
\end{equation}   
   
Suppose that the conditions \ref{eq:trans} are satisfied. Then the integer $1$ can be moved by cycle $\si$ to any integer $a\in \lsem1,p\rsem$, and afterwards it can be moved to any integer $b\in\lsem p+1, n\rsem \subseteq \supp(\ga_1\ga_2\hdots\ga_r)$ by means of a cycle $\ga_i$:
   $$
   \textrm{if } \ga_i^m(a)=b, \textrm{ then } \tau^m(a)=b.
   $$ 
Hence, the group $\gp{\si,\tau}$ is transitive.
   
   Conversely, suppose that the conditions \ref{eq:trans} are not satisfied. If some integer from $p+1$ to $n$ is not in the support of $\tau$, then it can't be sent to $1$ by means of the permutations $\si$ and $\tau$. If there is a cycle $\ga_i$ containing only integers from the interval $\lsem p+1, n\rsem$, then again those integers don't lie in the orbit of $1$ for the action of the group $\gp{\si,\tau}$. 
\end{proof}

\begin{proof}[Proof of Theorem] 
Let us show that for any permutation $\mu\in A_n$ moving at least $p+2$ points, there exist two permutations $\sigma$, $\tau\in S_n$ satisfying the following four conditions:
\renewcommand{\labelenumi}{\arabic{enumi})} %
\begin{enumerate} 
	\item $\mu=[\si,\tau]$,
	\item $\si$ is a $p$-cycle, 
	\item the pair $(\si,\tau)$ generates a transitive subgroup of $S_n$\hspace{1pt},
	\item $\tau$ can be chosen even or odd through multiplication by a transposition.
\end{enumerate}
The first two conditions are fulfilled by applying Proposition \ref{prop:bertram} with $\ell=p$: we have $\mu=[\si,\tau]$ for a $p$-cycle $\si\in A_n$. We may suppose that $\si = (1\;2\;\hdots\;p)$, otherwise conjugate $\si,\tau,\mu$ by a suitable permutation. 
\begin{enumerate}
   \item[3)] Consider the decomposition of $\tau$ into disjoint non-trivial cycles:
   \begin{equation}\label{eq:taucycles} \tau=\ga_1 \ga_2 \hdots \ga_r.\end{equation}
   We may suppose that each cycle $\ga_i$ of $\tau$ moves at least one integer in the set $\lsem 1, p\rsem$. Otherwise remove all cycles permuting only integers from $p+1$ to $n$\,, this will not affect the commutator $[\si,\tau]$. Since
   \begin{equation}\label{eq:mucycles}
	\mu = \si (\tau\si\inv\tau\inv) = (1\;2\;\hdots\;p)\cdot (\tau(p)\;\hdots\;\tau(2)\;\tau(1))
   \end{equation}
   and $|\supp(\mu)|>p$, there is an integer $x\in \lsem 1, p\rsem$ which is sent by $\tau$ to an integer $y\in \lsem p+1,n\rsem$ say $\ga_1(x)=y$. Denote by $y_1, y_2,\hdots,y_s$ the integers from $p+1$ to $d$ which are fixed by $\tau$,
   $$\{y_1, y_2,\hdots,y_s\}=\fix(\tau)\cap \lsem p+1,n\rsem\,,$$
   and let $\tau'=\tau\cdot (y\;y_1\; y_2\;\hdots\;y_s) = (A\;x\;y\;y_1\;\hdots\;y_s)\cdot\ga_2\hdots\ga_r$, where\footnote{A capital letter ($A$, $B$, \ldots) within a cycle abbreviates a sequence of integers. For example, $(1\;2\;3) = (A\;3)$ where $A = 1, 2$.} $\ga_1=(A\;x\;y)$. Then we have
   $$[\si, \tau']=\si \tau\cdot  (y\;y_1\;\hdots\;y_s)\si\inv (y\;y_1\;\hdots\;y_s)\inv\cdot\tau\inv=\si \tau \si\inv \tau\inv=[\si,\tau],$$
   and the group $\gp{\si,\tau'}$ is transitive by Lemma \ref{lem:sttransitive}. We take $\tau=\tau'$ in what follows.
   
   \item[4)] We have just obtained two permutations $\sigma\in A_n$ and $\tau\in S_n$ satisfying the conditions $1)$, $2)$ and $3)$. From (\ref{eq:mucycles}) and $|\supp(\mu)|\ge p+2$ we see that at least two distinct integers $x, x'\in \lsem 1, p\rsem$ are sent by $\tau$ to some integers $y, y'\in \lsem p+1,n\rsem$. In terms of a cycle decomposition (\ref{eq:taucycles}), we are dealing with situations of two kinds:
    $$
\renewcommand{\arraystretch}{1.4}
    \begin{array}{lc}
    \textrm{either } & \ga_1(x)=y\quad\textrm{and}\quad \ga_1(x')=y' \\
    \textrm{or } & \ga_1(x)=y\quad\textrm{and}\quad \ga_2(x')=y'
    \end{array}
\renewcommand{\arraystretch}{1}
    $$    
    In the first situation, we have $\ga_1=(A\;x\;y\;B\;x'\;y')$ and so
    $$\tau\cdot (y\;y')=(A\;x\;y)(B\;x'\;y')\cdot\ga_2\ga_3\hdots\gamma_s.$$
    In the second situation, we have $\ga_1=(A\;x\;y)$, $\ga_2=(B\;x'\;y')$ and so
    $$\tau\cdot (y\;y')=(A\;x\;y\;B\;x'\;y')\cdot\ga_3\hdots\gamma_s.$$
    So far, the commutator hasn't changed:
   $$[\si, \tau(y\;y')]=\si \tau\cdot  (y\;y')\si\inv (y\;y')\cdot\tau\inv=\si \tau \si\inv \tau\inv=[\si,\tau]\,,$$
    and\, $\gp{\si, \tau(y\;y')}$ is still a transitive subgroup of $S_n$ due to Lemma \ref{lem:sttransitive}. Since the permutations $\tau$ and $\tau (y\;y')$ have distinct parities, one of them is even (denote it by $\tau_1$) and the other is odd (denote it by $\tau_2$).
\end{enumerate}

Let us now show that permutations $\si$, $\tau\in S_n$ satisfying the conditions $2)$ and $3)$ above actually generate the entire group $A_n$ or $S_n$ depending on whether $\tau$ is even or odd respectively. Due to Lemma \ref{lem:primitivelargecycle}, the group $G=\gp{\si,\tau}$ is primitive since it is transitive and $p\ge \left[\frac{3n}{4}\right] > \frac{n}{2}$. According to Jordan's theorem (Proposition \ref{th:jordan}), a primitive group containing a $p$-cycle with $p\le n-3$ is either alternating or symmetric. So is $G$. 

Therefore, the pairs $(\sigma,\tau_1)$ and $(\sigma,\tau_2)$ are as required in the statement of our theorem.
\end{proof} 

\section{Corollaries}
\begin{Corollary}\label{cor:d10}
\emph{a)} When $n\ge 20$, except possibly for $n=32$, every permutation $\mu\in A_n$ that fixes at most $\left[\frac{n}{10}\right]-1$ points is the commutator of a generating pair of $A_n$ and a generating pair of $S_n$.\vss

\emph{b)} Given a real number $r > 4$, there exists an integer $N(r)$ such that for all $n\ge N(r)$, any permutation $\mu\in A_n$ fixing at most $\left[\frac{n}{r} \right]$ points is the commutator of a generating pair of $A_n$ and a generating pair of $S_n$.
\end{Corollary}
\begin{proof} 
a) Take an integer $n\ge 20$ not equal to $32$. It follows from (\ref{eq:primep}) and Table \ref{table:primenumber} that there exists a prime $p$ such that $\left[\frac{3n}{4}\right]\le p\le n - \left[\frac{n}{10}\right] -1$. If a permutation $\mu\in A_n$ fixes at most $\left[\frac{n}{10}\right]-1$ points, then it moves at least $n-\left[\frac{n}{10}\right]+1\ge p+2$ points. The required statement is, thus, a consequence of Theorem \ref{th:alternatingorigami}.

\vs b) If $r>4$ then the number $h=\frac{4}{3}\cdot\frac{r-1}{r}$ is greater than 1. Take a real $\varepsilon >0$ satisfying $h-\varepsilon >1$. For any integer $n$ large enough, there exists a prime $p$ between $\left[\frac{3n}{4}\right]$ and $(h-\varepsilon)\left[\frac{3n}{4}\right]$. Indeed, for each $k>1$ and $x$ large enough, there always is a prime between $x$ and $k x$\,: due to the asymptotic estimate $\pi(x)\sim \drob{x}{\ln{x}}$\; of the prime-counting function, we have $\pi(k x) - \pi(x)\underset{x\to\infty}{\longto} \infty$. Therefore,
$$\left[\frac{3n}{4}\right]\le\; p\;\le \left(\frac{4}{3}\cdot\frac{r-1}{r}-\varepsilon\right)\left[\frac{3n}{4}\right] \le \frac{r-1}{r}n - \varepsilon\left[\frac{3n}{4}\right] = n - \frac{n}{r} - \varepsilon\left[\frac{3n}{4}\right].$$
In particular, $p\le n-\left[\frac{n}{r}\right] - 2$ when $[3n/4]\ge 2/\varepsilon$. Thus for sufficiently large $n$ any permutation $\mu\in A_n$ that fixes at most $\left[\frac{n}{r}\right]$ points, moves at least $n-\left[\frac{n}{r}\right]\ge p+2$ points. According to Theorem \ref{th:alternatingorigami}, such a permutation $\mu$ is the commutator of a generating pair of $A_n$ and $S_n$.
\end{proof}

\vs\begin{Corollary}\label{cor:altrandomcomm}
The probability that a random element of the alternating group $A_n$ is the commutator of a generating pair of $A_n$ and a generating pair of $S_n$ tends to $1$ as $n\to\infty$.
\end{Corollary}
\begin{proof}
The number of permutations in $S_n$ that fix at least $k$ points doesn't exceed  $\binom{n}{k} (n-k)!=\frac{n!}{k!}$. Indeed, for given $k$ points there are $(n-k)!$ permutations fixing them, and there are $\binom{n}{k}$ ways to choose $k$ points among $n$. 
 
The number of permutations in $A_n$ fixing at most $k-1$ points is not less than $\frac{n!}{2} - \frac{n!}{k!}$. Therefore, due to Corollary \ref{cor:d10}a,  the probability $p_n$ that a random element of $A_n$ is the commutator of a generating pair of $A_n$ and a generating pair of $S_n$ is bounded by
$$p_n\ge \frac{\frac{n!}{2} - \frac{n!}{ [n/10]! }}{n!/2}=1 -  \frac{2}{ [n/10]! 
} \quad \textrm{for } n\ge 33.$$ 
This implies that\; $\lim\limits_{n\to\infty} p_n=1$, as required. 
\end{proof}

\begin{rema} 
It is well known and easy to prove by inclusion-exclusion principle, that the number of permutations in $S_n$ fixing no point (such permutations are called \deff{derangements}) is equal to 
\begin{equation}\label{eq:derang}
   a_n=n! \sum_{m=0}^n \frac{(-1)^m}{m!}=\left[\frac{n!}{e}+\frac{1}{2}\right].
\end{equation}
Moreover, denote by $b_n$ and $c_n$ the numbers of even and odd derangements respectively. Then\footnote{See, for instance, the note \cite[2005]{benjamin-b-n} by Arthur T. Benjamin, Curtis T. Bennett and Florence Newberger.}
\begin{equation}\label{eq:evenderang}
   b_n - c_n=(-1)^{n-1}(n-1).
\end{equation}
Indeed, let $A=(x_{i j})_{i,j=1}^{n}$ be the $n\times n$ matrix with zeroes on the diagonal and ones elsewhere, that is, $x_{i i}=0$ and $x_{i j}=1$ for any $1\le i\neq j\le n$. We have
$$\det{A}=\sum_{\la\in S_n}\sign{\la}\cdot x_{1\la(1)}x_{2\la(2)}\hdots x_{n\la(n)},$$
where each even derangement contributes $1$, each odd derangement contributes $-1$ and the permutations that fix a point contribute $0$. It is easy to show that the eigenvalues of the symmetric matrix $A$ are $-1$ with multiplicity $n-1$ and $n-1$ with multiplicity $1$. Thus, we obtain that $\det{A}=(-1)^{n-1}(n-1)$, implying the relation (\ref{eq:evenderang}).

From (\ref{eq:derang}), (\ref{eq:evenderang}) and $a_n=b_n+c_n$ follows the formula
$$b_n=\frac{1}{2}\left[\frac{n!}{e}+\frac{1}{2}\right] +\frac{(-1)^{n-1}(n-1)}{2}.$$
So, the number of even permutations fixing at most $k$ points is equal to
$$\binom{n}{0}b_n+\binom{n}{1}b_{n-1}+\hdots+\binom{n}{k}b_{n-k}=\frac{1}{2}\sum_{i=0}^k\binom{n}{i}
\left(  \left[\frac{(n-i)!}{e}+\frac{1}{2}\right]  -  (-1)^{n-i}(n-i-1) \right).$$
\end{rema}

\vs
Let $G$ be a group and let $(g_1,\,\hdots,\, g_k)$ be an ordered $k$-tuple of its elements. The following operations are called \deff{elementary Nielsen moves}:
\begin{enumerate}
  \item[(N1)] \hspace{1cm} interchange $g_i$ and $g_j$ where $i\neq j$;
  \item[(N2)] \hspace{1cm} replace $g_i$ by $g_i\inv$;
  \item[(N3)] \hspace{1cm} replace $g_i$ by $g_i g_j$ or $g_j g_i$ where $i\neq j$.
\end{enumerate}
A composition of such elementary operations is a \deff{Nielsen transformation}. Of course, it carries any generating vector (a tuple of elements that generates $G$) to a generating vector. In particular, a Nielsen transformation of a free basis of the free group $F_k$ defines an automorphism of $F_k$\,: with a Nielsen transformation sending $(x_1, \dots, x_k)$ to $(y_1, \dots, y_k)$ we associate the automorphism $(x_1, \dots, x_k)\mapsto (y_1, \dots, y_k)$. It is well-known that the automorphism group $\aut(F_k)$ consists of such automorphisms.

Let $G$ be a finite group, $d(G)$ be the minimal number of generators in $G$, and for $k\ge d(G)$ denote by $V_k(G)$ the set of $k$-tuples generating $G$. 
The group $\aut(F_k)$ acts on this set by Nielsen transformations, and the orbits are just the Nielsen equivalence classes. More accurately, if we fix a free basis $(x_1, \dots, x_k)$ of $F_k$, then $V_k(G)$ can be identified with the set of epimorphisms $\mathrm{Epi}(F_k, G)=\{f : F_k\twoheadrightarrow G\}$. From this point of view, the \emph{left} action of $\aut(F_k)$ on $V_k(G)$ is given by composition
\begin{equation*}
   \gamma\cdot f = f\circ \gamma\inv \quad \textrm{for}\; \gamma\in \aut(F_k)\; \textrm{and}\; f\in \mathrm{Epi}(F_k, G).
\end{equation*}
Now, consider also the diagonal action of the group $\aut(G)$ on $V_k(G)$. This action commutes with that of the group $\aut(F_k)$, as we have
\begin{equation*}
   \varphi\cdot(\gamma\cdot f) = \varphi\circ f \circ \gamma\inv \quad \textrm{for}\; \varphi\in\aut(G),\; \gamma\in \aut(F_k)\; \textrm{and}\; f\in \textrm{Epi}(F_k, G).
\end{equation*}
The orbits of the product $\aut(F_k)\times \aut(G)$ acting on the set $V_k(G)$ are called the \deff{systems of transitivity} (or \deff{$T_k$-systems}) of $G$. They were first introduced by Bernhard and Hanna Neumann in \cite[1951]{neumann-neumann}, and extensively studied since (see Igor Pak's survey \cite[2001]{pak}).

Let $G$ be a two-generator group and $(g_1,g_2)$ be a generating pair. The automorphism group $\aut(F_2)$ acts on the set $V_2(G)$ via composition of Nielsen moves (N1)--(N3). It is easy to check that if $(g_1,g_2)$ is sent to $(h_1,h_2)$ by a Nielsen move, then $[g_1,g_2]$ must be conjugated to $[h_1,h_2]^{\pm 1}$. This was first discovered by Graham Higman. Furthermore, if $G$ is one of the groups $A_n$ or $S_n$ where $n\neq 6$, then all automorphisms of $G$ are induced by conjugation through an element of $S_n$ (see Theorem 8.2A and Exercise 8.2.2 of the textbook \cite{dixon-mortimer} by John D. Dixon and Brian Mortimer). We obtain the following known lemma: 

\begin{Lemma}\label{lem:invariant}
	The conjugacy class of the commutator is an invariant of the $T_2$-systems in the alternating and symmetric groups of degree $n\neq 6$.
\end{Lemma} 

\vs In Table \ref{table:alternatingorigami}, we give the number of $T_2$-systems in $A_n$ for small values of $n$. These data principally come from the following papers:
\begin{itemize}
   \item[] \hspace{-.5cm}\cite[1951]{neumann-neumann} by Bernhard H. Neumann and Hanna Neumann, and \cite[1936]{hall} by Philip Hall (for $n=5$),
   \item[] \hspace{-.5cm}\cite[1972]{stork} by Daniel Stork (for $n=6$),
   \item[] \hspace{-.5cm}\cite[1998]{higuchi-m} by Osamu Higuchi and Izumi Miyamoto (for $n=7, 8, 9$).
\end{itemize}

\begin{table}[htdp]
\caption{Number of $T_2$-systems in $A_n$ with $3\le n\le 9$.}
\renewcommand{\arraystretch}{1.2}
\begin{equation*}
\begin{array}{|c|c|c|c|c|}
\hline
\textbf{$G$} & \textbf{Order}&
	\begin{array}{c} \textbf{Number of conjugacy classes}\\ \textbf{of commutators of gen.~pairs} \end{array} & 
	\begin{array}{c} \textbf{Number of}\\ \textbf{$T_2$-systems} \end{array} & 
	\begin{array}{c} \textbf{Number of $\aut(G)$-}\\ \textbf{orbits in $V_2(G)$} \end{array} \\ \hline\hline 
	A_3 & 3 & 1 & 1 & 4\\
	A_4 & 12 & 1 & 1 & 4\\
	A_5 & 60 & 2 & 2 & 19\\
	A_6 & 360 & 3 & 4 & 53\\
	A_7 & 2520 & 6 & 16 & 916\\
	A_8 & 20160 & 7 & 18 & 7448\\
	A_9 & 181440 & 10 & 38 & 77004\\\hline
\end{array}
\end{equation*}
\renewcommand{\arraystretch}{1}
\label{table:alternatingorigami}
\end{table}%

According to \cite[Theorem 1.8]{garion-shalev}, for any $\varepsilon >0$ and $n$ large enough, the number of $T_2$-systems in the group $A_n$ is at least $e^{(1/2-\varepsilon)\log^2 n}$. Now, our theorem allows us to improve this asymptotic expansion:
\begin{Corollary}\label{cor:alternatingstrata} 
For any real $r>4$ and sufficiently large $n$, the number of $T_2$-systems in $A_n$ and the number of $T_2$-systems in $S_n$ are both greater than
$$
\frac{1}{2}P(n)-\frac{1}{2}{\medmu P\big(n-[\sfrac{n}{r}]\big)},
$$
where $P(n)$ denotes the number of unrestricted partitions\footnote{$P(n)$ is the number of ways of writing $n$ as a sum of positive integers; two sums that differ only in the order of their summands are considered to be the same partition.} of a positive integer $n$.
\end{Corollary}
\begin{proof} 
Due to Lemma \ref{lem:invariant} for $n\neq 6$, the number of $T_2$-systems in $A_n$ is not less than the number of conjugacy classes of the commutators $[\si,\tau]$, where $(\sigma,\tau)$ generate $A_n$ (the same is true for $S_n$). Due to Corollary \ref{cor:d10}b, this includes all conjugacy classes of even permutations fixing at most $\left[\frac{n}{r}\right]-1$ points for sufficiently large $n$.

According to J\'ozsef D\'{e}nes, Paul Erd\H{o}s and Paul Tur\'{a}n \cite[1969]{denes-e-t}, the number $c(n)$ of conjugacy classes in the alternating group $A_n$ is equal to\footnote{From this equation they derived the expansion
$c(n)\sim \frac{1}{2}P(n)\sim \frac{1}{8n\sqrt{3}} e^{\pi\sqrt{\frac{2n}{3}}}$\; as $n\to\infty$.} 
\begin{equation}\label{eq:cd}
   c(n)=\frac{1}{2}P(n)+ \frac{3}{2}Q(n),
\end{equation}
where $Q(n)$ denotes the number of partitions of $n$ into distinct odd summands (see also \cite[1980]{girse} by Robert D. Girse). Notice that the number of conjugacy classes in $A_n$ containing permutations that fix at least $k$ points is clearly $c(n-k)$. Therefore, the number of conjugacy classes  in $A_n$ of permutations fixing at most $k-1$ points is equal to $c(n)-c(n-k).$
For $k=\left[\frac{n}{r}\right]$ this gives 
$$c(n)-c\big(n-[\sfrac{n}{r}]\big)\;\; \textrm{classes}.$$

It is easy to see that $Q(n)>Q\big(n-[\sfrac{n}{r}]\big)$ for sufficiently large $n$ (see also Lemma \ref{lem:altqm} below). From this and the relation (\ref{eq:cd}), we deduce a lower bound for  the number of conjugacy classes of permutations in $A_n$ fixing at most $[\sfrac{n}{r}]-1$ points:
$$c(n)-c(n-[\sfrac{n}{r}])
=\frac{1}{2}\bigg(P(n)-P(n-[\sfrac{n}{r}])\bigg)+ \frac{3}{2}\bigg(Q(n)-Q(n-[\sfrac{n}{r}])\bigg)
> \frac{1}{2}\bigg(P(n)-P(n-[\sfrac{n}{r}])\bigg),$$
when $n$ is large enough, as required.
\end{proof}

\noindent \emph{\textbf{Remark 1.}}
A famous result of Godfrey H. Hardy and Srinivasa Ramanujan\index{Theorem!Hardy and Ramanujan} \cite[1918]{hardy-ramanudjan} states that 
$$P(n)\sim \frac{1}{4n\sqrt{3}} e^{\pi\sqrt{\frac{2n}{3}}}\quad\textrm{as }n\to\infty.$$
From this asymptotic behaviour follows that 
\begin{align*} 
\lim\limits_{n\to\infty} \frac{\medmu P\big(n-[\sfrac{n}{r}]\big) }{P(n)} &= 
\frac{r}{r-1}  \lim\limits_{n\to\infty} e^{\pi\sqrt{\frac{2}{3}} \left( \sqrt{n-[\sfrac{n}{r}]} -\sqrt{n} \right) } = 0, \\
\quad \textrm{and so}\quad &
P(n)-{\medmu P\big(n-[\sfrac{n}{r}]\big)} \sim \frac{1}{4n\sqrt{3}} e^{\pi\sqrt{\frac{2n}{3}}} \quad\textrm{as }n\to\infty.
\end{align*}
Therefore, we obtain a lower bound
$$
\frac{1}{2}\bigg(P(n)-{\medmu P\big(n-[\sfrac{n}{r}]\big)}\bigg)
\ge \frac{C}{n} e^{\pi\sqrt{\frac{2n}{3}}},
$$
for any $n> r$ and some real constant $C>0$ small enough.

\vs
\noindent \emph{\textbf{Remark 2.}} 
In the proof of the corollary above, we used the fact that $Q(n)-Q\big(n-[\sfrac{n}{r}]\big)>0$ for sufficiently large values of $n$. Actually, we even have the next stronger property:

\begin{Lemma}\label{lem:altqm}
For $n\ge 3$, the function $Q$ is non-decreasing: $Q(n)\ge Q(n-1)$.
\end{Lemma}
\begin{proof}
The generating function for the sequence $(Q(n))_{n\in\N}$ is
$$f(x)=\prod_{l=1}^{\infty}(1+x^{2l-1})= (1+x)(1+x^3)\hdots (1+x^{2l-1})\hdots = 1+\sum_{n\in\N} Q(n)x^n.$$
Denote by $R(n)$ the number of partitions of $n$ into distinct odd summands greater than $2$. Then
$$\prod_{l=2}^{\infty}(1+x^{2l-1}) = 1+\sum_{n\in\N} R(n)x^n,$$
and so we obtain
\begin{align*} 
1+ \sum_{n=2}^{\infty} \big(Q(n)-Q(n-1)\big)x^n &= (1-x) f(x)\\
&=(1-x^2)\prod_{l=2}^{\infty}(1+x^{2l-1})=1-x^2 + \sum_{n=3}^{\infty} \big(R(n)-R(n-2)\big)x^n,
\end{align*}
as $Q(1)=1$ and $R(1)=R(2)=0$. Therefore, the relation $Q(n)-Q(n-1)=R(n)-R(n-2)$ is satisfied for any $n\ge 3$. Let us show that $R(n)\ge R(n-2)$. Indeed, if 
$$n-2=a_1+a_2+\hdots+a_r$$
is a partition of $n-2$ with all $a_i$'s odd and such that $a_1> a_2> \hdots>a_r\ge 3$, then we get a partition of $n$ having the same property:
$$n=(a_1+2)+a_2+\hdots+a_r.$$
This gives an injection from the set of partitions for $n-2$ with distinct odd summands $\ge 3$  into the set of such partitions for $n$. As a result, $Q(n)-Q(n-1)=R(n)-R(n-2)\ge 0$.
\end{proof}

\newpage
\addcontentsline{toc}{section}{References}
\bibliographystyle{siam}

\end{document}